\documentclass[twoside,leqno]{article}

\usepackage[letterpaper]{geometry}

\usepackage{siamproceedings}

\usepackage{color}
\usepackage[T1]{fontenc}
\usepackage{amsfonts}
\usepackage{amssymb}
\usepackage{graphicx}
\usepackage{epstopdf}
\usepackage{enumitem}
\usepackage{algorithmic}
\ifpdf
  \DeclareGraphicsExtensions{.eps,.pdf,.png,.jpg}
\else
  \DeclareGraphicsExtensions{.eps}
\fi

\newsiamremark{remark}{Remark}
\newsiamthm{counterexample}{Counterexample}
\newsiamremark{hypothesis}{Hypothesis}
\crefname{hypothesis}{Hypothesis}{Hypotheses}

\newsiamthm{claim}{Claim}
\crefname{claim}{Claim}{Claim}
\newsiamthm{fact}{Fact}
\crefname{fact}{Fact}{Fact}

\newsiamremark{question}{Question}
\newsiamremark{problem}{Problem}

\usepackage{amsopn}

\newcommand\Diff{\operatorname{Diff}}
\newcommand\Prob{\mathbb P}
\newcommand\Proberg{\Prob_{\rm erg}}
\newcommand\hTOP{h_{\rm TOP}}
\newcommand\htop{h_{\rm top}}
\newcommand\RR{\mathbb R}
\newcommand\MME{\operatorname{MME}}
\newcommand\Probhyp{\Prob_{\rm hyp}}

\newcommand\HC{\operatorname{HC}}
\newcommand\homrel{\stackrel{HC}\sim}

\begin{document}

\newcommand\relatedversion{}
\renewcommand\relatedversion{\thanks{The full version of the paper can be accessed at \protect\url{https://arxiv.org/abs/0000.00000}}} 

\title{{\Large Chaos on surfaces and beyond: a new notion of dynamical hyperbolicity}\\$ $\\
\small (text prepared for the International Congress of Mathematicians 2026)}
    \author{J\'er\^ome Buzzi\thanks{C.N.R.S. \& Universit\'e Paris-Saclay (\email{jerome.buzzi@universite-paris-saclay.fr}, \url{https://imo.universite-paris-saclay.fr/~buzzi).}}
}

\date{}

\maketitle


\fancyfoot[R]{\scriptsize{Preprint version}}





\begin{abstract}
We present some developments in the study of chaotic dynamics following the solution of a conjecture of Newhouse on the measures maximizing the entropy of smooth surface diffeomorphisms.  We focus on \emph{strong positive recurrence}, a generalization of the classical Anosov-Smale theory of uniform hyperbolicity introduced in a joint work with Sylvain Crovisier and Omri Sarig. This new property is general enough to be satisfied by all smooth surface diffeomorphisms with positive entropy, yet it still ensures many quantitative properties such as exponential mixing or limit theorems for regular functions. We also present some open problems, including its abundance (or not) in higher dimensions.
\end{abstract}

\noindent
{\it Keywords and phrases:} smooth ergodic theory; symbolic dynamics; entropy; Lyapunov exponent; homoclinic class; measure maximizing the entropy; strong positive recurrence; exponential mixing; almost sure invariance principle; effective intrinsic ergodicity.

\medbreak

\noindent
{\it 2020 Mathematics Subject Classification.} 37C40, 28D20, 37A25, 37B10, 37D25, 37D35, 37E30.

\section{Introduction.}
This text is devoted to the presentation of some recent developments in the ergodic theory of hyperbolic dynamics for a general mathematical readership with a passing familiarity with dynamics. 
We present investigations motivated the solution \cite{BCS-MME} of Newhouse conjecture about measures maximizing the entropy for smooth surface diffeomorphisms. 

Our focus will be on \emph{Strong Positive Recurrence} (or SPR) for diffeomorphisms, a new notion of hyperbolicity  introduced in \cite{BCS-SPR} and involving hyperbolicity, Lyapunov exponents, and Kolmogorov-Sina\u\i\  entropy.

Unless otherwise specified, the results are joint with Sylvain Crovisier and Omri Sarig. I do not claim any exhaustivity and I apologize to the many colleagues whose important work I omitted.

\subsection{Invariant measures and entropies.}
We recall some basic terminology and introduce our notations, referring to \cite{Katok-Hasselblatt} for definitions and background (see also \cite{Downarowicz-book} about dynamical entropies).
 
For a Borel map $f:M\to M$, $\Prob(f)$ denotes the set of $f$-invariant Borel probability measures. If the space $M$ is a compact metric space, then $\Prob(f)$ equipped with the weak topology (that is, the weak star topology induced by the continuous functions on $M$) is also compact and metrizable.

The idea of entropy is fundamental in dynamics. It gives rise to many notions which can be defined following the same general scheme. One first defines an entropy at some scale $\epsilon>0$ as the growth exponent $\limsup_{n\to\infty} \frac1n\log r(\epsilon,n)$ of some numbers $r(\epsilon,n)$ of ``relevant'' orbit segments of length $n$ that can be distinguished at the scale $\epsilon$. The entropy is then $h=\lim_{\epsilon\to0} h(\epsilon)$. This scheme can be applied to the two following classical dynamical entropies. 

The \emph{topological entropy} $\htop(f)$ for a continuous map $f$ is obtained by considering \emph{all} orbit segments to be relevant in the previous scheme. Focusing on sets of orbit segments representative with respect to some measure, one obtains the \emph{Kolmogoro-Sinai  entropy} $h(f,\mu)$ for endomorphisms $f$ of probability spaces defined by a measure~$\mu$. Given a Borel map $f$ on some Borel space $M$, this induces a function $h(f,\cdot):\Prob(f)\to[0,\infty]$. One considers the top entropy and the measures achieving it (called \emph{measures maximizing the entropy} or MMEs):
 $$
    \hTOP(f):= \sup_{\nu\in\Prob(f)} h(f,\nu) \text{ and }\MME(f) := \left\{\mu\in\Proberg(f) : h(f,\mu) =\hTOP(f) \right\}
 $$
(one restricts to $\Proberg(f)$ since a measure has maximal entropy if and only if almost all its ergodic components have).

When $f$ is  continuous on a compact metric space $M$, the \emph{variational principle} identifies the supremum $\hTOP(f)$ with the topological entropy $\htop(f)$.
Thus one can understand measures maximizing the entropy as ``seeing the whole complexity of the map'' and they often satisfy some  equidistribution properties.  
A theorem of Newhouse \cite{Newhouse-Continuity} ensures that all $C^\infty$ maps $f$ on compact manifolds have $\MME(f)\ne\emptyset$ (even though this may fail in finite differentiability \cite{Misiurewicz-NoMax,Buzzi-NoMax}).

\medskip

\paragraph{Standing assumptions and notations.} 
$M$ is a manifold of dimension $d\ge2$. When we want to recall its dimension we write $M^d$. All manifolds are compact, boundaryless, and equipped with a Riemannian structure. We denote by $\|\cdot\|$ the corresponding norms on the tangent spaces $T_xM$, $x\in M$, as well as the induced operator norms. The set of $C^r$-diffeomorphisms of $M$ is written as $\operatorname{Diff}^r(M)$. It is equipped with the usual topology ($r\ge1$ and often $r=\infty$; when $r=k+\alpha$ with $0<\alpha<1$, then the diffeomorphisms are $k$ times continuously differentiable and its $k$th order differential is H\"older-continuous with exponent $\alpha$). We write $C^{1+}$ or $\Diff^{1+}(M)$ to denote the $C^{1+\alpha}$ maps or diffeomorphisms for some unspecified $\alpha$. A subset $X\subset M$ is invariant under a map $f$ if $f(X)=X$. We can then define the restriction $f:X\to X$ denoted by $f|_X$.

\subsection{Outline of the paper.} 
We start in Section 2 by recalling the Anosov-Smale theory of uniform hyperbolicity and its motivation by  attempts to describe most or typical dynamics. We then review some more recent notions of hyperbolicity, so that we can introduce Strong Positive Recurrence (SPR) in a natural way.

In Section 3, we will discuss SPR for smooth surface diffeomorphisms, explaining the relation with a continuity property of exponents proved in \cite{BCS-ECLE}.

In Section 4, we introduce the key notion of Borel homoclinic class from \cite{BCS-MME}. It is at this level that one has a general uniqueness result of the MME and it is also there that one is able to characterize finely the SPR property and build appropriate codings. 

This will lead us to the ``symbolic ancestor'', i.e., the SPR property for Markov shifts which gives its name to the new property and also is the main tool that allow us to explore the consequences in Section 5.

In Section 6 we present some well-studied classes of diffeomorphisms that we can be easily shown to be SPR (based on previous works).
 
Finally in section 7, we select some open problems that we believe to be relevant to the development of the theory.

\section{Hyperbolicity.}
In this section, we start by recalling the problem of structural stability for dissipative diffeomorphisms of compact manifolds and how it motivated the study of hyperbolicity. We define various types of dynamical hyperbolicity, first following  the classical works of Anosov and Smale, then giving some elements of the theory of Pesin and finally of Young towers. We see how they  naturally lead to Strong Positive Recurrence (or SPR).

\subsection{Structural stability.}
Andronov and Pontryagin argued in a celebrated 1937 paper \cite{Andronov-Pontryagin1937}  that, since one cannot exactly know the laws of evolution of physical systems, the diffeomorphisms $f:M\to M$ modelling them should be chosen \emph{structurally stable}: using some relevant topology, any small perturbation $g$ of $f$ remains topologically conjugate to $f$, i.e., there is a homeomorphism $h:M\to M$ such that $f = h^{-1}\circ g\circ f$.

A weaker notion is defined in terms of the chain recurrent set $\mathcal R(f)$, i.e., the set of points such that, for every $\epsilon>0$, there is an $\epsilon$-pseudo-orbit from $x$ to $x$, see  \cite[\S2.3]{Robinson-book}. One says that $f$ is \emph{$C^r$ $\mathcal R$-stable} if for every $C^r$ small perturbation $g$ of $f$ in the space of all diffeomorphisms, the restrictions $f|_{\mathcal R(f)}$ and $g|_{\mathcal R(g)}$ are topologically conjugate.

Such questions of stability were revived by Smale \cite{Smale1967} who, for a short period of time, conjectured \cite{Smale-Mexicana} that typical dissipative dynamical systems should have a very simple dynamics to a finite number of hyperbolic periodic orbits  and in particular be structurally stable.

\subsection{Lyapunov exponents and Pesin blocks.}
Loosely speaking, a diffeomorphism $f$ is \emph{nonuniformly hyperbolic} when all (relevant) tangent vectors $v\in TM$ admit non-zero \emph{(pointwise) Lyapunov exponents}:
 $$
    \lambda(f,v):=\lim_{n\to\infty} \frac1n\log\|Df^n.v\| \ne 0 \quad
    (\|\cdot\| \text{ denotes the Euclidean norms induced by the Riemannian structure}).
 $$
The above limits do not exist everywhere, but the multiplicative ergodic theorem of Oseledets \cite{Katok-Hasselblatt} ensures their existence for all $v\in T_xM$ for almost every point $x\in M$ with respect to any invariant probability measure (the Oseledets-regular points). More precisely, this theorem gives a basis $(v_1,\dots,v_d)$ of $T_xM$ and $d$ numbers $\lambda^1(f,x)\ge\lambda^2(f,x)\ge\dots\ge\lambda^d(f,x)$ such that $\lambda^i(f,x)=\lim_{n\to\pm\infty} \frac1n\log\|Df^n.v_i\|$. For $\mu\in\Prob(f)$, we consider the \emph{average exponents}:
 $$
   \lambda^i(f,\mu) := \int_M \lambda^i(f,x)\, d\mu(x) \qquad (i=1,\dots,d)
 $$
When $\mu$ is ergodic, $\lambda^f(i,\mu)$ is the value of $\lambda^i(f,x)$ for $\mu$-a.e. $x\in M$.

The \emph{unstable space} at $x\in M$ is the vector subspace $E_x^+\subset T_xM$ generated by all $v_i$'s with $\lambda(f,v_i)>0$. The \emph{stable space} $E_x^-$ is defined similarly by the $v_i$'s with $\lambda(f,v_i)<0$. These spaces depend measurably on $x\in M$.

In general, the limit defining the Lyapunov exponent is not uniform with respect to the base point of the tangent vector $v$.
A measure of this (lack of)  uniformity  is to consider how the following sets fill the space:

\begin{definition}
A \emph{Pesin block} is a nonempty Borel set $\Lambda\subset M$ such that,
for some numbers $0<\epsilon<\chi$ and some number $C>0$, the following hold for every $x\in\Lambda$:
 $$
   \exists E\oplus F = T_xM \quad
    \forall k\in\mathbb Z\;\forall n\in\mathbb N\;
       \max\left(\|Df^n|Df^k(E)\|, \|D_xf^{-n}|Df^k(F)\|\right) \le C e^{\epsilon |k|} e^{-n\chi}.
 $$
We denote by $\Lambda(\chi,\epsilon,C)$  the set of all points $x\in M$ satisfying the above.
\end{definition}

In other words, points in a given Pesin block, though they may fail to have well-defined exponents, satisfy $(C,\chi)$-uniform expansion/contraction estimates. Moreover this uniformity only worsens $\epsilon$-slowly along the orbits.

\medbreak

We note that any Pesin block is contained in a standard Pesin block $\Lambda(\chi,\epsilon,C)$. Moreover,  each $\Lambda(\chi,\epsilon,C)$ is compact and is approximately invariant: $f^{\pm1}(\Lambda(\chi,\epsilon,C))\subset \Lambda(\chi,\epsilon,e^\epsilon C)$.
We also note that the splitting $E\oplus F$ is unique when it exists since $0<\epsilon<\chi$. In fact $E=E_x^-$ and $F=E_x^+$ if $x$ is Oseledets-regular. 

\subsection{Anosov-Smale uniform hyperbolicity.}

Inspired by key examples such as geodesic flows  with negative curvature, Anosov, Smale, Newhouse and others introduced notions of uniform hyperbolicity. We note the following two (see, e.g., \cite{Robinson-book}):

\begin{definition}
The diffeomorphism $f$ is \emph{Anosov} if the whole manifold $M$ is a Pesin block.
It is \emph{$\mathcal R$-hyperbolic}  if its chain recurrent set is a Pesin block. 
\end{definition}

Let us give three examples. 
(1) Compactify the action of $x\mapsto 2x$ on $\mathbb R^d$ to obtain a uniformly hyperbolic diffeomorphism of the $d$-sphere whose limit set is reduced to the two hyperbolic fixed points corresponding to $0$ (repelling) and to $\infty$ (attracting). 
(2) Quotient the action of $(x,y)\mapsto(2x+y,x+y)$ by $\mathbb Z^2$ to obtain an Anosov diffeomorphism on the torus $\mathbb R^2/\mathbb Z^2$. 
(3) Smale constructed ``by hand'' the so-called \emph{Smale horseshoe}, a uniformly hyperbolic diffeomorphism  on the sphere whose limit set is a Cantor set (up to finitely many hyperbolic periodic orbits).
See \cref{fig:ex-unif-hyp}.

\begin{figure}[htbp]
  \centering
  \includegraphics[width=0.2\columnwidth]{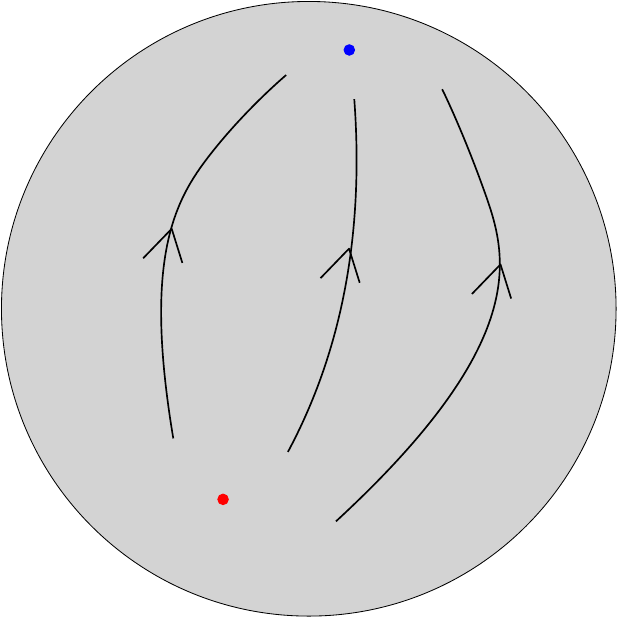}
  \hspace{0.8cm}
  \includegraphics[width=0.11\columnwidth]{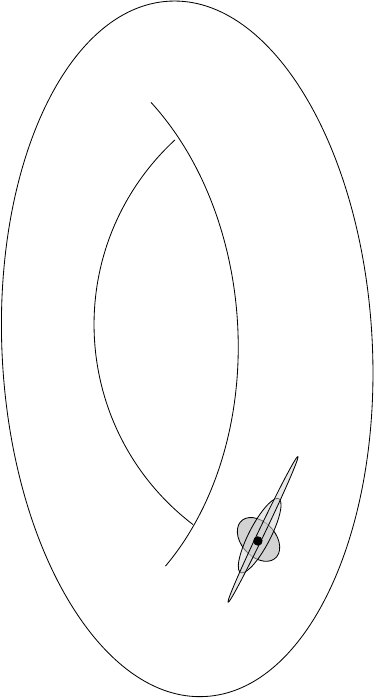}
  \hspace{0.8cm}
  \includegraphics[width=0.2\columnwidth]{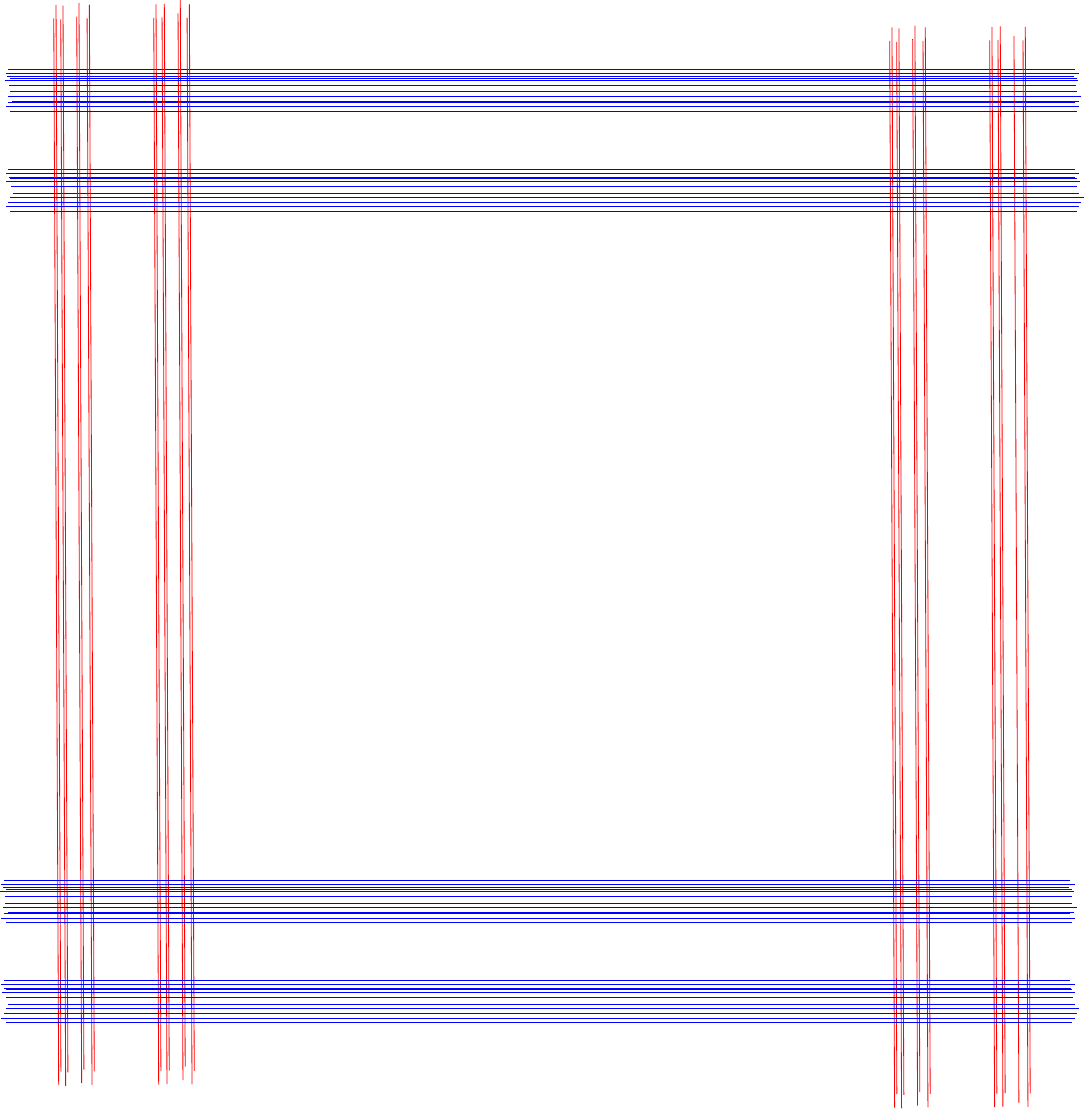}
  \caption{From left to right: (1) Morse-Smale dynamics on a sphere with two fixed points (a  source and a sink); (2) Toral automorphism with a few iterates of a small ball around $0$; (3) Smale's horsehoes with its local stable (horizontal, blue) and unstable (vertical red) manifolds.}
  \label{fig:ex-unif-hyp}
\end{figure}

Uniform hyperbolicity is a classical and rather complete theory. Such dynamics\footnote{To avoid triviality we assume from now on that the recurrent set is infinite.}  exhibit very rich behavior which is both chaotic at the individual level and strongly structured at the global level.
For instance, on the one hand, individual orbits are ``strongly instable'': there is $\epsilon_0>0$ such that, for any two distinct points $x,y$ in $\mathcal R(f)$, $d(f^n(x),f^n(y))>\epsilon_0$ for some $n\in\mathbb Z$. Moreover, there is an exponentially large number of well-separated orbit segments in the sense that th topological entropy $\htop(f)$ is positive.

On the other hand, the Artin-Mazur zeta functions
 \begin{equation}\label{eqZeta}
     \zeta_f(z) := \exp \left( \sum_{n\ge1} \frac{z^n}{n} \#\operatorname{Fix}(f,n)\right)
     \text{ where } \operatorname{Fix}(f,n) \text{ is the set of isolated points in }\{x\in M: f^n(x)=x\} 
  \end{equation}
which in general can be purely formal power series \cite{Kaloshin-zeta},
have, in the uniformly hyperbolic case, analytic extension to rational functions (implying very strong constraints on the  number of periodic orbits).

\medskip

Palis and Smale \cite{Palis-Smale-1970} conjectured and Robinson \cite{Robinson-1976} demonstrated   stability theorems. Even more strikingly, uniform hyperbolicity and structural stability coincide \cite{Mane-Stability} (see \cite[\S9.9 and \S10.4]{Robinson-book} for this version):

\begin{theorem}[Ma\~n\'e]
A diffeomorphism of a compact manifold is $\mathcal R$-hyperbolic if and only if it is $\mathcal R$-stable with respect to the $C^1$ topology. 
\end{theorem}

Thus, a structurally stable dynamics can be very complex. But the {\it coup de grâce} was still to come:

\begin{theorem}[Abraham-Smale \cite{Abraham-Smale-1970},  Newhouse \cite{Newhouse-PSPM-1970}, Bonatti-Diaz-Pujals \cite{Bonatti-Diaz-Pujals-Sinks}]
The set of uniformly hyperbolic diffeomorphisms is not dense in $\Diff^r(M)$ as soon as ($r\ge2$ and $\dim M\ge 2$) or ($r\ge1$ and $\dim M\ge3$).
\end{theorem}

There are dynamics that cannot be made structurally stable by a small perturbation! So even to study generic dynamics (or just a dense set of diffeomorphisms), uniformly hyperbolic dynamics is not enough. This, together with a number of important examples, has motivated the study of hyperbolic-like dynamics ``beyond uniform hyperbolicity''.
 We refer the interested reader to \cite{BDV-book-2005,Crovisier-ast}.

\subsection{Pesin nonuniform hyperbolicity.}
This point of view is rooted in smooth ergodic theory, i.e., one studies the dynamics of typical points with respect to some invariant measure $\mu\in\Prob(f)$.

\begin{definition}[Pesin]
An invariant probability measure $\mu\in\Prob(f)$ is \emph{(nonuniformly) Pesin-hyperbolic} if, for $\mu$-a.e. $x\in M$, for all $v\in T_xM\setminus0$, $\lambda(f,v)\ne0$. It is \emph{$\chi$-hyperbolic} for some $\chi>0$ if, for $\mu$-a.e. $x\in M$, for all $1\le i\le d$, $|\lambda^i(f,x)|>\chi$. Moreover, it is of \emph{saddle type} if there are both positive and negative exponents at $\mu$-a.e. $x\in M$.
\end{definition}

Let us stress that  Pesin-hyperbolicity is a property of an invariant measure, not of the diffeomorphism. Note also that $\mu\in\Prob(f)$ is Pesin-hyperbolic if and only if almost all of its ergodic components are Pesin-hyperbolic.

It is usually a difficult question to exclude zero exponents. The following fundamental result allows to deduce hyperbolicity from  entropy for surface diffeomorphisms:

\begin{theorem}[Ruelle-Margulis]\label{thmRuelleInequ}
Let $f$ be a $C^1$ diffeomorphism of the compact manifold $M$. For any $\mu\in\Prob(f)$,
 $$
    h(f,\mu) \leq  \int_M \sum_{i=1}^d \max(\lambda^i(f,x),0)\, d\mu(x).
$$
In particular, if $M$ is a surface and $\mu$ is ergodic, then  $h(f,\mu)>0$ implies that $\mu$ is $\chi$-hyperbolic of saddle-type for any $0<\chi<h(f,\mu)$.
\end{theorem}

As the product of a positive entropy diffeomorphism with a circle rotation shows, the last assertion does not hold in dimension $3$ or higher.

\medbreak

The starting point of Pesin theory is the following construction of smooth invariant manifolds which are the basic tools for studying many problems (local ergodicity, Markov partitions, spectral decompositions to name a few). See \cite{Katok-Hasselblatt} for  background.

\begin{theorem}[Pesin stable manifold]\label{thmPesinStable}
Let $f\in\Diff^r(M)$ with $r>1$ and $M$ a closed manifold. Assume that $\mu\in\Prob(f)$  is hyperbolic. Then for $\mu$-a.e. $x\in M$, the set
 $$
    W^s(f,x) := \left\{y\in M : \limsup_{n\to+\infty}\frac1n\log d(f^ny,f^nx)<0 \right\}
 $$
is a $C^r$ immersed manifold with tangent space $T_xW^s(f,x)$ coinciding with the stable Oseledets space:
 $
   E^s_x:=\{v\in T_xM: \limsup_{n\to+\infty}\frac1n\log\|Df^n.v\|<0\}.
 $
Moreover, the map $x\mapsto W^s(f,x)$ is measurable in the $C^r$ compact-open topology.
\end{theorem}

The \emph{Pesin unstable manifolds} are $W^u(f,x):=W^s(f^{-1},x)$. Their properties follow from the above theorem.

\medbreak

A hallmark of Pesin theory is the need that the diffeomorphism is at least $C^{1+\alpha}$ with $\alpha>0$ (see \cite{BCS-Pesin}).
\Cref{thmPesinStable} can be deduced from the next two results about Pesin blocks:

\begin{proposition}\label{propPesinBlock}
Given a $\chi$-hyperbolic measure $\mu$, for any $0<\epsilon<\chi$,
 $
   \lim_{C\to\infty} \mu\left(\Lambda(\chi,\epsilon,C)\right)=1.
 $
\end{proposition}

\begin{lemma}
For any parameters $0<\epsilon\ll\chi$ and $C\ge1$:
(i) for any $x\in\Lambda(\chi,\epsilon,C)$, $W^s(f,x)$ is a $C^1$ immersed manifold; (ii)~$x\in\Lambda(\chi,\epsilon,C)\longmapsto W^s(f,x)$ is continuous in the compact-open topology.
\end{lemma}

\subsection{Stronger notions of nonuniform hyperbolicity.}
Pesin hyperbolicity (the nonvanishing of the Lyapunov exponents) yields a number of deep results. For instance, Pesin has shown that the ergodic components of hyperbolic smooth invariant measures are themselves smooth, hence countably many \cite{Barreira-Pesin-book-2023}. Sarig has shown similarly that the ergodic components of measures maximizing the entropy on surfaces are also countably many when the topological entropy is positve \cite{Sarig-JAMS}. However, Pesin hyperbolicity is unsufficient for many common and important properties such as existence of the relevant measures and more quantitative properties such as speeds of mixing, large deviation principles, or limit theorems (see below).

A natural idea is to try and restore some uniformity in the hyperbolicity by iterating suitable powers of the transformation, i.e., trying to rewrite the large iterate $f^n$ as a composition $f^{n_r}\circ\dots\circ f^{n_1}$ such that (i) each iterate $f^{n_i}$ is almost uniformly hyperbolic; (ii) the iterates $n_i$ are not too large. Sometimes this can be done by taking the first return map to a carefully selected subset, but often one needs to select ``good returns''. L.-S.~Young \cite{YoungTower1,YoungTower2} has proposed a very successful scheme of this type with respect to some natural reference measure such as the volume. 

Such \emph{Young towers} (as they are called now) are a very powerful tool that has been adapted to many important and diverse nonuniformly hyperbolic dynamics. We refer to the book \cite{Alves-book-2020} and the references therein for an uptodate presentation and their numerous and important applications and consequences. However, finding a good subset on which to define the good returns, as well as the exact specification of these good returns, remains a delicate problem.

\subsection{Definition of Strong Positive Recurrence.} 
\cite{BCS-SPR} defines Strong Positive Recurrence for a diffeomorphism on an invariant measurable subset $X$ in terms of its ergodic measures  with \emph{nearly} maximal entropy:

\begin{definition}\label{def-diffeo-SPR}
A diffeomorphism $f$ of a compact manifold $M$ is said to be \emph{Strongly Positively Recurrent} (\emph{SPR} for short) on some invariant measurable subset $X$ if for some $\chi_0>0$, for any $\epsilon>0$, there are numbers  $h_0<\hTOP(f|_X)$, $C_0,m_0>0$ such that:
 $$
   \forall \mu\in\Proberg(f|_X)\quad h(f,\mu)>h_0\implies \mu(\Lambda(\chi_0,\epsilon,C_0))\ge m_0.
 $$
We say that $f$ is SPR, when it is SPR on $X=M$.
\end{definition}

This generalizes uniform hyperbolicity. Indeed, for an Anosov or Axiom-A diffeomorphism, there is a single Pesin block that carries all invariant measures. 

\medbreak

The SPR property implies a wealth of consequences for measures maximizing the entropy when $X=M$ or when $X$ is ``reasonable'' (see below).

\begin{remark}\label{remSPRpot}
The SPR property above only allows the study of measures maximizing the entropy. However, it extends naturally to \emph{equilibrium states.} 
Those are the measures maximizing the so-called pressure $p_\phi(f,\mu):=h(f,\mu)+\int_M\phi\, d\mu$ for some upper-bounded Borel function $\phi:M\to\RR\cup\{-\infty\}$ (see \cite{Ruelle-Thermo} for background). 

Indeed,  \cite[****]{BCS-SPR} defines the notion of $\phi$-SPR  and shows that, for H\"older or geometric functions $\phi$, this notion implies the same consequences for the equilibrium states as SPR for measures maximizing the entropy, e.g., exponential mixing, limit theorems,\dots Establishing the $\phi$-SPR property for interesting functions $\phi$ is another matter and is discussed below.
\end{remark}

\begin{remark}
It is interesting to compare SPR and Pesin hyperbolicity. 

First, checking SPR involves in principle not only the measures maximizing the entropy, but all nearly maximal ones. See however \cref{thmSPRfromExpRet} on surfaces.

Second, SPR is strictly stronger than the Pesin-Hyperbolicity of every ergodic measure with maximal or nearly maximal entropy, see \cref{coroNoMaxNoSPR}. In fact, the Pesin-hyperbolicity of all ergodic measure with nearly maximal entropy does not imply the properties implied by SPR, see \cref{propBadHypMME}.
\end{remark}

We believe that the interest of the above \cref{def-diffeo-SPR} will be supported by its many consequences (existence of finitely many MMEs with exponential mixing and other good statistical properties) as well as its many characterizations:
 \begin{enumerate}
     \item existence of a common Pesin block seen by all large entropy measures, i.e., \cref{def-diffeo-SPR};
   \item convergence of Lyapunov exponents (see \cref{thmSPRECLEsurf}  for surfaces and \cref{thm-EHEC} beyond);
   \item symbolic dynamics being SPR in the sense of Gurevich and Sarig (\cref{secSPRCoding});
   \item exponential tail of the return time (\cref{thmSPRfromTail});
 \end{enumerate}
Moreover, we will see that characterizations (1) and (2) present \emph{rigidity phenomena}: they are equivalent to apparently stronger properties of the same type.

\section{Strong Positive Recurrence of surface diffeomorphisms.}

This section is devoted to the SPR property for surface diffeomorphisms. First and foremost, all $C^\infty$-diffeomorphisms of surfaces are SPR, giving the strongest argument for the interest of this notion. Then our result take their simplest and strongest form. More general and precise results considering irreducible pieces and their period in any dimension will be discussed in later sections.

\subsection{SPR for all.}
Perhaps the most striking fact about SPR  is how general it is on surfaces. We have already remarked that for $C^1$ surfaces diffeomorphisms, all positive entropy ergodic measures are Pesin-hyperbolic, see \cref{thmRuelleInequ}.

\begin{theorem}[B-Crovisier-Sarig \cite{BCS-SPR}]
\label{thmSurfaceSPR}
Let $f:M\to M$ be a diffeomorphism of a compact surface. If $f$ is $C^\infty$ smooth and $\htop(f)>0$, then $f$ is strongly positively recurrent.
\end{theorem}

There are many surface diffeomorphisms that are SPR but not uniformly hyperbolic. Indeed, if $(f_t)_{t\in[0,1]}$ is an isotopy\footnote{More precisely, a continuous family of $C^\infty$ surface diffeomorphisms.} with $\htop(f_0)\ne \htop(f_1)$, take $f_t$ for any of the uncountably many $t\in[0,1]$ such that $e^{\htop(f_t)}$ is not an algebraic number (such $t$'s exist by continuity of $t\mapsto \htop(f_t)$).

We note that \cite{Berger-Henon} had already proved exponential mixing and central limit theorem for the unique measure maximizing the entropy of ''strongly regular'' H\'enon maps (a set of parameters with positive Lebesgue measure) by constructing an SPR Young tower.

\begin{remark}
One cannot drop any of the assumptions as shown by the following examples: 
(i) $\tau:\mathbb T^2\to\mathbb T^2$ an arbitrary translation of the $2$-torus;
(ii) $F={\tiny \begin{pmatrix}2 & 1\\1 & 2\end{pmatrix}}\times\rho$ where the integer matrix acts on $\mathbb T^2$ (with positive entropy) and $\rho$ is a rotation on the circle;
(iii) for any $1<r<\infty$ and any surface $M$, there are  $C^r$ diffeomorphisms  with $\htop(f)>0$ but not SPR, see \cref{coroNoMaxNoSPR}.

Indeed, (i) and (ii) do not have any Pesin blocks because of the existence of an isometric direction.
\end{remark}

\begin{remark}
Zero entropy and SPR are not exclusive, though the intersection is somewhat trivial. Indeed,  a zero entropy diffeomophism is SPR if and only if it has finitely many ergodic measures, each defined by some hyperbolic periodic orbit. Its non-wandering set can of course be more complicated (e.g., contain separatrices as in the classical Bowen example, see \cite{Takens-Bowen}).
\end{remark}

In the rest of this section, we explain how to derive \cref{thmSurfaceSPR} from a continuity property of the Lyapunov exponents and, conversely, give further consequences for the Lyapunov exponents.

\subsection{Continuity of Lyapunov exponents on surfaces.}
Proving the SPR property for smooth surface diffeomorphisms led us to establish a continuity property of Lyapunov exponents. Using Kingman's subadditive theorem and $C^1$ smoothness, it is easy to see that the average top Lyapunov exponent $\lambda^1(f,\mu)$
is an upper semicontinuous function of $\mu\in\Prob(f)$ with respect to the weak topology. However, it is \emph{not} lower semicontinuous (a phenomenon related to homoclinic tangencies). So it is perhaps surprising that the following holds for $C^\infty$ surface diffeomorphisms:

\begin{theorem}[B-Crovisier-Sarig \cite{BCS-ECLE}]\label{thmECLE-htop-surf}
Let $f$ be a $C^\infty$-smooth diffeomorphism of a compact surface. Let $\nu_k\in\Proberg(f)$ weakly converge to some  $\mu\in\Prob(f)$. 
\smallbreak

If $\lim_k h(f,\nu_k)=\htop(f)>0$, then $\lim_k \lambda^i(f,\nu_k) = \lambda^i(f,\mu)$  for $i=1,2$. 
\end{theorem}

\begin{remark}\label{remBothC0}
By Oseledets theorem, $\lambda^1(f,\nu)+\lambda^2(f,\nu) = \int \log |\det Df|\, d\nu$ which is a continuous function of $\nu\in \Prob(f)$ by the definition of the weak topology. Therefore the continuity of one exponent is equivalent to the continuity of both.
\end{remark}

In fact, we proved the following more general property:

\begin{theorem}[B-Crovisier-Sarig \cite{BCS-ECLE}]\label{thmECLE}
Let $f$ be a $C^\infty$-smooth diffeomorphism of a compact surface. Let $\nu_k\in\Proberg(f)$ weakly converge to some  $\mu\in\Prob(f)$. Assume that the  limits $\lim_k h(f,\nu_k)$, $\lim_k \lambda^1(f,\nu_k)$ exist and are  positive. Then, one can write $\mu=(1-\beta)\mu_0+\beta\mu_1$ with $0<\beta\le 1$, $\mu_0,\mu_1\in\Prob(f)$, and:
 $$
    \lim_k h(f,\nu_k) \leq \beta \cdot h(f,\mu_1) \text{ and }
    \lim_k \lambda^1(f,\nu_k) = \beta\cdot \lambda^1(f,\mu_1)\;.
 $$ 
\end{theorem}

Let us comment briefly on \Cref{thmECLE} and its  somewhat delicate proof.

\medbreak
\paragraph{A relative version of Ruelle inequality?} It is tempting to compare it with Ruelle's inequality, \cref{thmRuelleInequ}. Indeed, specializing to ergodic limit measure $\mu$, we have the following corollary: {\it Let $\nu_k\in\Proberg(f)$ weakly converge to some $\mu$ assumed to be ergodic. Assume that the limits $\lim_k h(f,\nu_k)$, $\lim_k \lambda^1(f,\nu_k)$ exist and are positive. Then:}
 $$
    \limsup_k \frac{h(f,\nu_k)}{\lambda^1(f,\nu_k)} \le \frac{h(f,\mu)}{\lambda^1(f,\mu)}.
 $$
This proves the upper semicontinuity of the Hausdorff dimension over $\Proberg(f)$.

This may look like a version of \cref{thmRuelleInequ}.
However, the settings are different: Ruelle's inequality holds for all $C^1$-diffeomorphisms whereas the above fails in $C^1$ smoothness.

\medbreak
\paragraph{An asymmetric, suboptimal bound.}
Replacing $f$ by $f^{-1}$ may change the decomposition $\mu=(1-\beta)\mu_0+\beta\mu_1$. In particular, the bound on entropy may fail to be optimal. This asymmetry is also a difficulty in trying to maximize the Hausdorff dimension of invariant measures (see \cite{Buzzi-IMPAN-2026}).

\medbreak
\paragraph{Ingredients of the proof.}
Underlining the difference between this continuity property and Ruelle's inequality, we note that  its proof involves more sophisticated tools: first Pesin theory and especially Yomdin theory. The last in particular requires high smoothness.

\medbreak
We refer the interested reader to the original paper \cite{BCS-ECLE}. The lecture notes \cite{Buzzi-IMPAN-2026} review the techniques of the proof  of  \cref{thmECLE} and present it in details. Burguet \cite{Burguet-Decomposition} has extended this decomposition at the topological entropy to finite smoothness ($C^r$ if $\htop(f)>\lambda(f)/r$ where $\lambda(f)$ is the logarithm of the Lipschitz constant). These techniques can be applied to $3$-dimensional diffeomorphisms or flows \cite{Zang-arxiv-2025} since the entropy can still be read off curves. When there are critical points, new difficulties appear and they have been considered for interval maps \cite{Li-arxiv-2024}. 

\subsection{Surface SPR and continuity of Lyapunov exponents.}
On surfaces, the SPR property turns out to be \emph{equivalent} to the convergence of exponents. 

\begin{theorem}[B-Crovisier-Sarig \cite{BCS-SPR}]\label{thmSPRECLEsurf}
Let $f$ be a diffeomorphism of a compact surface with positive entropy. If $f$ is $C^{1+}$-smooth then $f$ is SPR if and only if the following holds:

(*) there is $\chi>0$ such that, for any sequence of measures $\nu_k\in\Proberg(f)$ with some weak limit $\mu$, if $\lim_k h(f,\nu_k)=\htop(f)>0$, then $\mu$ is $\chi$-hyperbolic and the exponents: $\lim_k\lambda^i(f,\nu_k)=\lambda^i(f,\mu)$ for each $i=1,2$. 
\end{theorem}

\begin{proof}[Proof of \cref{thmSurfaceSPR}]
Taking $\lim_k h(f,\nu_k)=\htop(f)>0$, we get the continuity of the exponent immediately from \cref{thmECLE-htop-surf} to get $\lim_k \lambda^i(f,\nu_k)=\lambda^i(f,\mu)$ for $i=1$. By \cref{remBothC0}, this gives it also for $i=2$. 

We then check that the limit $\mu$ is $\chi$-hyperbolic for $\chi:=\htop(f)/2>0$ independent of $(\nu_k)_{k\ge1}$. Since $f$ is $C^\infty$ smooth, the entropy is upper semicontinuous by a result of Newhouse \cite{Newhouse-Continuity} so $h(f,\mu)=\hTOP(f)$, hence almost every ergodic component of $\mu$ is an MME so $\mu$ is $\chi$-hyperbolic for any $0<\chi<\htop(f)$.

We may now apply \cref{thmSPRECLEsurf}.
\end{proof}

\begin{remark}
C. Luo and D. Yang \cite{luo-yang-2025-uscEntropy} have shown that \cref{thmSPRECLEsurf} remains true if one removes the condition that $\mu$ is $\chi$-hyperbolic. They deduce from the convergence of the exponents that, even in $C^{1+}$-smoothness, the limit $\mu$ is a measure of entropy $\htop(f)>0$. Thus it is $\htop(f)/2$-hyperbolic by the same argument as above.
\end{remark}

We now outline the proof of \cref{thmSPRECLEsurf} (note that it is the restriction of \ref{thm-EHEC} to surfaces).
\medbreak

The SPR property follows from the convergence (*) by a \emph{purely $C^1$ argument}, that we now sketch. To prove SPR, we need to find a subset of the tangent bundle $TM$  which has lower bounded measure and is  expanded (or contracted) by some fixed iterate $Df^N$. The Oseledets theorem provides an iterate $N$ which depends on the measure. However, one can deduce from the convergence of the exponents in (*)  that the measures (lifted to the unstable section $x\mapsto E_x^+$) actually converge. This gives a subset of $M$ with measure close to $1$ and an expanded direction. One can do the same for the contraction and take the intersection. Using that the resulting subset has measure close enough to $1$, a combinatorial argument (based on Pliss lemma, or, equivalently, the maximal ergodic inequality) yields the slow decay and therefore the Pesin block. This proves SPR, and in fact slightly more:
\medbreak
{\sl (ET) There is $\chi>0$ such that, for every $\epsilon,\gamma>0$, there are $h<\htop(f)$ and $C>0$ satisfying: 
 $$
   \forall\mu\in\Proberg(f)\quad h(f,\mu)>h \implies \mu\left(\Lambda(\chi,\epsilon,C)\right) > 1-\gamma.
 $$
}
\medbreak
In \cite{BCS-SPR},  (ET) is called entropy-tightness.
\medbreak

The converse implication, that SPR implies the convergence (*), is proved in \cite{BCS-SPR} by relying on symbolic dynamics (see below). The construction of the coding however requires Pesin theory and therefore $C^{1+}$ smoothness. We believe that this is only technical:

\begin{question}
Do all SPR $C^1$-diffeomorphisms on compact surfaces satisfy the continuity property (*) ?
\end{question}

\subsection{Effective intrinsic ergodicity.}
The following phenomenon was discovered recently for uniformly hyperbolic dynamics by Kadyrov \cite{Kadyrov-2015}. We will see that this extends to the nonuniformly hyperbolic, SPR setting, but also yields a striking \emph{rigidity property}: if the Lyapunov exponents converge, then (the diffeomorphism is SPR and therefore) the exponent must converge at an explicit speed. 

\begin{theorem}[B-Crovisier-Sarig \cite{BCS-SPR}]\label{thmExFinBernoulli}
Let $f\in\Diff^{1+}(M^2)$ which is SPR with a unique measure maximizing the entropy $m$. Then there is $C>0$ such that
 $
  \forall\mu\in\Prob(f)\quad \left|\lambda^1(f,m)-\lambda^1(f,\mu)\right| \leq C\sqrt{h(f,m)-h(f,\mu)}.
 $
\end{theorem}

The proof of this result relies on the coding, where such a result was established by Ruhr and Sarig.

\medbreak

We give two  corollaries.
The first one strengthens \cref{thmECLE-htop-surf} by giving the speed of the convergence (under the same assumptions):

\begin{corollary}
Let $f\in\Diff^\infty(M^2)$ with $\htop(f)>0$. For any $\nu_k\in\Proberg(f)$ weakly converging to $m\in\Prob(f)$ with $\lim_k h(f,\nu_k)=\htop(f)$, then, for all $i=1,2$: 
 $$
   \limsup_k  \frac{|\lambda^i(f,m)-\lambda^i(f,\nu_k)|}{\sqrt{h(f,m)-h(f,\nu_k)}}<\infty.
 $$
\end{corollary}

The second one can be understood as a \emph{rigidity phenomenon}:

\begin{corollary}
Let $f\in\Diff^{1+}(M^2)$ with $\htop(f)>0$. Assume that, for any $\nu_k\in\Proberg(f)$ weakly converging to $m\in\Prob(f)$ with $\lim_k h(f,\nu_k)=\htop(f)$, 
 $$
    |\lambda^i(f,m)-\lambda^i(f,\nu_k)| \stackrel{k\to\infty}\longrightarrow 0
 $$
then, for any such sequence $\nu_k$ and limit $m$, as $k\to\infty$,
 $$
    |\lambda^i(f,m)-\lambda^i(f,\nu_k)| = \mathcal O(\sqrt{h(f,m)-h(f,\nu_k)}).
 $$
\end{corollary}

\section{Homoclinic classes, Markov shifts, and codings.}
To express fully  our results, we need to consider the following notions  of irreducibility and periodicity and their MMEs and symbolic dynamics. This collects the constructions from a number of papers, including \cite{Sarig-JAMS,BenOvadia-JMD,BCS-MME,BCS-SPR}.

\subsection{Homoclinic classes.}
The  notion of homoclinic relation was introduced by Newhouse and adapted to the nonuniformly hyperbolic setting by \cite{RRTU-transitive}.
Let $\Probhyp(f)\subset\Proberg(f)$ be the set of Pesin hyperbolic ergodic measures.
If $A,B$ are two immersed submanifold, $A\pitchfork B$ denotes the set of transverse intersection points (i.e., $x\in A\cap B$ such that $T_xA+T_xB=T_xM$).

\begin{definition}[B-Crovisier-Sarig]
Two measures $\mu_1,\mu_2\in\Probhyp(f)$ are \emph{homoclinically related} if for $\mu_1$-a.e. $x$ and for $\mu_2$-a.e. $y$,
 \begin{equation}\label{eq-homrel}
    W^u(f,x)\pitchfork W^s(f,f^k(y))\ne\emptyset \text{ and }
    W^s(f,x)\pitchfork W^u(f,f^\ell(y))\ne\emptyset \text{ for some integers }k,\ell.
 \end{equation}
We write $\mu_1\homrel\mu_2$ and call $\HC(\mu):=\{\nu\in\Probhyp(f):\nu\homrel\mu\}$ the \emph{homoclinic class} of $\mu$. We extend this notion to hyperbolic periodic orbits through the identification $\mathcal O\mapsto \delta_{\mathcal O}:=\frac1{\#\mathcal O}\sum_{p\in\mathcal O} \delta_p$, the unique invariant measure carried by $\mathcal O$.
\end{definition}

\newcommand\per{\operatorname{per}}

The following is immediate, but important:

\begin{fact}
Every measure $\mu$ in a homoclinic class $H$ is ergodic and hyperbolic. Moreover the \emph{unstable dimension}: $d^u(\mu):=\max\{1\le i\le d:\lambda^i(f,x)>0$ $\mu$-a.e.$\}$ is the same for all $\mu\in H$.
\end{fact}

Using Pesin theory and Katok's horseshoe theorem, one proves (see \cite{BCS-MME}):

\begin{lemma}
The homoclinic relation $\homrel$ is an equivalence relation on $\Probhyp(f)$ which extends the classical one between hyperbolic periodic orbits.  Moreover, every homoclinic class of measures contains a hyperbolic periodic orbit. 

Additionally, for any homoclinic class of measures $H\subset\Probhyp(f)$, there is an invariant Borel subset $X_H\subset M$ with the following properties: 
\begin{enumerate}
\item $\forall\mu\in\Proberg(f)$ $\mu\in H\iff\mu(X_H)=1$;
\item for any $x,y\in X_H$, the Pesin manifolds are well-defined at $x,y$ and \eqref{eq-homrel} holds;
\item if $y$ belongs to $X_{H'}$ for a distinct class $H'$, then \eqref{eq-homrel} does not hold
\end{enumerate}
In particular, the Borel homoclinic classes are pairwise disjoint and in bijection with the homoclinic classes of measures.
\end{lemma}

\begin{definition}
The \emph{period} of a Borel homoclinic class $X\subset M$ is:
 $
    \per(X) := \gcd(\{\#\mathcal O: \mathcal O$ a periodic orbit in $X\})$.
    Its \emph{entropy} is $h(X):=\hTOP(f|_X)$.
\end{definition}

It is convenient to set $\per(H):=\per(X_H)$ and $h(H):=h(X_h)$. The following is easy:

\begin{fact}
A homoclinic class of measures has zero entropy if and only if it  is reduced to a single periodic orbit.
\end{fact}

The following easy proposition shows how the global SPR property brings the focus on a finite number of homoclinic classes.

\begin{proposition}[B-Crovisier-Sarig]\label{propGlobalSPR}
Let $f$ be a $C^{1+}$-smooth diffeomorphism of a compact manifold. If $f$ is SPR if and only if there is $h_0<\hTOP(f)$ such that:
 \begin{enumerate}
  \item  $f$ admits finitely many Borel homoclinic classes with top entropy equal to $\hTOP(f)$;
  \item each of these classes with top entropy equal to $\hTOP(f)$  is SPR;
  \item all ergodic measures not carried by the these classes have entropy less than $h_0$.
\end{enumerate} 
\end{proposition}

We sketch the proof since it brings together Pesin blocks and the homoclinic relation.

\begin{proof}
Since $f$ is SPR, there exists a Pesin block $\Lambda$ and $h<\hTOP(f)$ such that every $\mu\in\Proberg(f)$ with $h(f,\mu)>h$ satisfies $\mu(\Lambda)>0$. However, on $\Lambda$ the stable and unstable manifolds are continuous, in particular in their directions and scales (the scale at which they are almost flat). Therefore there exists $\epsilon>0$ such that if $x,y\in\Lambda$ and $d(x,y)<\epsilon$ then $W^u(x)\pitchfork W^s(y)\ne\emptyset$ and $W^s(x)\pitchfork W^u(y)\ne\emptyset$. In particular, any subset of $\Lambda$ with diameter less than $\epsilon$ cannot meet more than one homoclinic class. The compactness of $\Lambda$ is enough to conclude.
\end{proof}

\subsection{SPR Markov shifts}\label{secSPRcoding}
We rely on the codings developed by Sarig \cite{Sarig-JAMS} for surface $C^{1+}$-diffeomorphisms, generalized in arbitrary dimension by Ben~Ovadia \cite{BenOvadia-JMD} and localized to Borel homoclinic classes in \cite{BCS-MME}.

\medbreak
\paragraph{Markov shifts.} 
The symbolic dynamics is given by Markov shifts, a generalization of the subshifts of finite type that code for uniformly hyperbolic diffeomorphisms (we refer to \cite{Gurevich-Savchenko,Sarig-ETDS-1999,Sarig-CMP-2001} for background).

\newcommand\toG{\stackrel{\mathcal G}\to}
\begin{definition}
A simple oriented graph $\mathcal G\subset\mathcal V\times\mathcal V$ on a countable set $\mathcal V$ defines the \emph{Markov shift} $(\Sigma(\mathcal G),\sigma)$ defined by: (1) the space $\Sigma(\mathcal G):=\{v\in\mathcal V^{\mathcal Z}:\forall n\in\mathbb Z\; (v_n,v_{n+1})\in\mathcal G\}$; (2) the shift map $\sigma:\Sigma(\mathcal G)\to\Sigma(\mathcal G)$, $(x_n)_{n\in\mathcal Z}\mapsto (x_{n+1})_{n\in\mathcal Z}$. We often use $\Sigma(\mathcal G)$ or just $\Sigma$ to represent the Markov shift $(\Sigma(\mathcal G),\sigma)$. The elements of $\mathcal V$ are called \emph{vertices} and one writes $a\toG b$ to mean that $(a,b)\in\mathcal G$.

\medbreak
$\Sigma$ is \emph{irreducible} if any two vertices $v,w\in\mathcal V$ can be joined by a nontrivial path, i.e.,   $(v_0,\dots,v_\ell)\in\mathcal V^{\ell+1}$ with $\ell\ge1$, $v_0=v\toG v_1\toG\dots\toG v_\ell=w$.
\end{definition}

We endow all Markov shifts with the distance $d(x,y):=\exp\left(-\inf\{|n|:x_n\ne y_n\}\right)$ on $\Sigma$ (it is compatible with the product topology). We need a few more definitions.

The \emph{regular part} of $\Sigma$ is:
 $$
    \Sigma^\# := \{ x\in\Sigma : \exists v,w\in\mathcal V\; \{k\le0:x_k=v\} \text{ and }
      \{\ell\ge0:x_\ell=w\} \text{ are both infinite}\}.
 $$
 
A Markov shift $\Sigma$ does not need to be compact so it may not have a well-defined topological entropy. Its top entropy $\hTOP(\sigma|_\Sigma)$ (or simply $\hTOP(\Sigma)$) is however well-defined. By a result of Gurevich, an irreducible Markov shift with finite top entropy has at most one measure maximizing the entropy. From now on, we  assume: $\hTOP(\Sigma)<\infty$. 

An irreducible Markov shift $\Sigma$ has a period $\per(\Sigma):=\gcd(\{n: \exists x\in\Sigma$ $\sigma^nx=x\})$. It is called \emph{aperiodic} if this period is $1$. 
The following property of Markov shifts was introduced by Gurevich  (under a slightly different name and using a different but equivalent characterization, see \cite{BCS-SPR}):

\begin{definition}
A Markov shift $\Sigma$ with finite top entropy is \emph{Strongly Positively Recurrent} (or SPR) if there is a finite subset $\mathcal V_*$ of the vertices and some numbers $h_*<\hTOP(\Sigma)$ and $\gamma_*>0$ such that
 $$
   \forall\mu\in\Proberg(\Sigma)\; h(\sigma,\mu)> h_* \implies \mu\left(\{v\in\Sigma:v_0\in\mathcal V_*\}\right) > \gamma_*.
 $$
\end{definition}

SPR  Markov shifts have been shown to have many strong properties and a large part of our strategy will consist in trying to push down (through $\pi$) those properties known for $\Sigma$ or for its one-sided counterpart $\Sigma^+$ to $f|X$. Let us quote two of the most important results on SPR Markov shifts from our point of view.

\medbreak

Thanks to Ornstein theory, many natural measures are Bernoulli in the following sense:

\begin{definition}\label{DefBernoulli}
A \emph{Bernoulli scheme} is the left-shift on $A^\mathbb Z$, where $A$ is a discrete probability space.  A measure $\mu\in\Prob(f)$ is \emph{Bernoulli up to a period} if $(f,\mu)$ is isomorphic to the product of a Bernoulli scheme and a cyclic permutation. The measure is \emph{Bernoulli} if the period of the cyclic permutation is $1$.
\end{definition}

\begin{theorem}[Gurevich \cite{Gurevich-Savchenko}]
If $\Sigma$ is an irreducible Markov shift with finite top entropy and the SPR property, then it admits exactly one measure maximizing the entropy $m$. Moreover, $(\sigma,m)$ is Bernoulli up to a period.
\end{theorem}

The following is a simplified version of Theorem B.7 in \cite{BCS-SPR} which we deduced from the spectral gap theorem of \cite{Cyr-Sarig}. To obtain a spectral gap on a space of functions, one has to pass to a non-invertible shift (which can be thought of as expanding rather than hyperbolic):
\medbreak

The \emph{one-sided counterpart} of $\Sigma$ is $\Sigma^+:=\{(x_n)_{n\ge0}:(x_n)_{n\in\mathcal Z}\in\Sigma\}$ together with $\sigma^+:(x_n)_{n\ge0}\mapsto(x_{n+1})_{n\ge0}$.
\medbreak

It will also be necessary to pass to the case of period $1$. These two reductions are well-known and the appropriate versions can be found in \cite{BCS-SPR}.

\newcommand\un[1]{\underline{#1}}
\begin{theorem}[Cyr-Sarig]
Let {$\Sigma^+$} be an irreducible one-sided SPR Markov shift
with finite Gurevich entropy and period $1$, let $\mu$ be the measure of maximum entropy.
For any $\beta>0$ there exist a linear operator $\mathcal L$ on a complex Banach space $(\mathcal L,\|\cdot\|_{\mathcal L})$ contained in $C(\Sigma^+)\cap L^1(\mu)$  such that:
\begin{enumerate}
\item $C^\beta(\Sigma^+)\subset \mathcal L$ with $\|\psi\|_{\mathcal L}\le \|\psi\|_\beta$ for all $\psi\in C^\beta(\Sigma^+)$;
\item $L$ is the Ruelle operator on $\mathcal L$: $(L\psi)(\un{x})=\sum_{\sigma(\un{y})=\un{x}}\psi(\un{y})$ for all $\psi\in\mathcal L$;
\item $L$ is transfer operator of $\mu$ on $\mathcal L$: $\sigma_*(\psi\cdot \mu)=L(\psi)\cdot\mu$ for all $\psi\in\mathcal L$;
\item $L$ has a spectral gap: $L=P+N$ with bounded operators on $\mathcal L$ such that $PN=NP=0$,  $P(\psi)=\int\psi\, d\mu\cdot 1$, and the spectral radius is $\rho(N)<1$.
\end{enumerate}
\end{theorem}

\subsection{Coding of homoclinic classes.}\label{secSPRCoding}
Building on  previous works \cite{Sarig-JAMS,BenOvadia-JMD,BCS-MME}, we showed in \cite{BCS-SPR} that codings of the type constructed by Sarig and Ben~Ovadia lift Pesin blocks to finite union of cylinders. We deduce from this that codings of SPR Borel homoclinic classes are SPR Markov shifts. We obtain:

\begin{theorem}[B-Crovisier-Sarig \cite{BCS-SPR}]
Let $f$ be a $C^{1+}$-diffeomorphism of a compact manifold $M$. Let $X$ be a Borel homoclinic class. For each parameter $\chi>0$,  there exists an irreducible Markov shift $\Sigma_\chi$ and a H\"older map $\pi_\chi:X\to M$ such that:
 \begin{enumerate}
  \item $\pi_\chi\circ\sigma = f\circ\pi_\chi$ on $\Sigma_\chi$;
  \item for all $x\in M$, $\pi_\chi^{-1}(x)\cap\Sigma_\chi^\#$ is finite;
  \item for any $\mu\in\Prob(f|_X)$ which is $\chi$-hyperbolic, $\mu(\pi_\chi(\Sigma_\chi^\#))=1$;
  \item for any $\nu\in\Prob(\Sigma_\chi)$, $h(f,(\pi_\chi)_*(\nu))=h(\sigma,\nu)$ and $\mu(X)=1$;
  \item if $f$ is SPR on the Borel homoclinic class $X$, then $\Sigma_\chi$ is SPR.
 \end{enumerate}
\end{theorem}

\begin{remark}
One can easily see  that either there is some $\chi_0>0$ such that all measures in the class $X$ are $\chi_0$-hyperbolic or the above does depend on $\chi$. In particular,  the H\"older exponent of $\pi:\Sigma\to M$ must tend to $0$ when $\chi$ does. From now on, we omit the $\chi$ from the notations $\Sigma_\chi,\pi_\chi$.
\end{remark}

The coding allow the following very general local result about MMEs:

\begin{corollary}[B-Crovisier-Sarig \cite{BCS-MME}]\label{CorLocalMME}
Let $f$ be a $C^{1+}$-diffeomorphism of a compact manifold $M$. Let $X$ be a Borel homoclinic class. 

Then $f|_X$ has at most one MME. If it exists it is Bernoulli up to a period which with the period $\per(X)$.

Moreover, if $X$ is SPR, then $\MME(f|_X)\ne\emptyset$.
\end{corollary}

\section{Consequences of SPR.}

\medbreak
\subsection{Global and qualitative SPR.}
We start with the following basic qualitative result.

\begin{theorem}[B-Crovisier-Sarig]\label{thmSPRfinBernoulli}
Let $f\in\Diff^{1+}(M)$ for some compact manifold $M$. If $f$ is SPR, then:
 \begin{enumerate}
  \item $\MME(f)$ is finite and nonempty and in bijection with the set of Borel homoclinic classes with top entropy equal to $\htop(f)$;
  \item each $m\in\MME(f)$ is Bernoulli up to a period and carried by the corresponding Borel homoclinic class which is SPR.
\end{enumerate} 
\end{theorem}

\begin{remark}
In the case of a $C^\infty$-diffeomorphism of a compact surface, this gives a new proof of the finiteness of $\MME(f)$ obtained in \cite{BCS-MME}. This new proof avoids the topological argument of \cite{BCS-MME}. However establishing SPR in \cite{BCS-ECLE} is more complicated than the original proof.
\end{remark}

This shows that some surface diffeomorphisms  cannot be SPR. Indeed, we built in \cite{Buzzi-NoMax} surface diffeomorphisms without measures of maximal entropy which are $C^r$ for any finite $r$. Hence:
 
\begin{corollary}\label{coroNoMaxNoSPR}
For every $1<r<\infty$, there is a $C^r$ diffeomorphism of a compact surface with positive entropy which is not SPR.
\end{corollary}

\begin{remark}\label{propBadHypMME}
By Ruelle inequality (\ref{thmRuelleInequ}), such diffeomorphisms  have all their ergodic measures with positive entropy  Pesin-hyperbolic, without being SPR.
\end{remark}

\subsection{Quantitative consequences of SPR}
We give three statement: exponential mixing, large deviation principle, and almost sure invariance principle (i.e., approximation by Brownian motion). We stress that the results are as good as those obtained for uniformly hyperbolic dynamics. 
We note that the statements below apply directly to diffeomorphisms of compact manifold with arbitrary dimension.

\medbreak
\paragraph{Exponential mixing.}

\medbreak

Given an SPR diffeomorphism of some compact manifold, if some measure maximizing the entropy has some mixing property, then it is \emph{exponentially mixing}:

\begin{theorem}[B-Crovisier-Sarig \cite{BCS-SPR}]
Let $f\in\Diff^{1+}(M)$ which is SPR on a Borel homoclinic class $X$. Let $m\in\MME(f|_X)$. If $m$ is totally ergodic, then, for all $\alpha>0$, there is $\kappa<1$ such that
 $$
   \forall u,v\in C^\alpha(M)\; \forall n\ge0\quad 
       \int_M u\circ f^n\cdot v\, dm - \int _M u\, dm \cdot \int _M v\, dm = \mathcal O(\kappa^n).
 $$
\end{theorem}

\begin{remark}
The above statement has a natural generalization in the case where $\per(X)>1$, see \cite{BCS-SPR}. This is the case for all statements below where we will assume that the period is $1$ for simplicity.
\end{remark}

\medbreak
\paragraph{Exponential tail for the MME.} 
Using a that the MME of an irreducible SPR Markov shift is ``exponentially filling'', we prove:

\begin{theorem}\label{thmSPRfromTail}
Let $f$ be a $C^{1+}$-diffeomorphism and $X$ be a Borel homoclinic class.  This class is SPR if and only if there are $\mu\in\Proberg(f|_X)$ and $\chi,\delta>0$ s.t.:
\begin{enumerate}
\item The measure $\mu$ is an MME for $f|_X$. \item For every $0<\epsilon<\chi$, there is $0<\theta<1$ such that, writing $\Lambda:=\Lambda(\chi,\epsilon,C)$,
    $
    \mu[\tau_\Lambda>n]=O(\theta^n)$ as $n\to\infty$, where $\tau_\Lambda(x):=\inf\{n\geq 0: f^n(x)\in \Lambda\}.
    $
\item
Every $\nu\in\Proberg(f|_X)$ with $h(f,\nu)>\hTOP(f|_X)-\delta$ is $\chi$-hyperbolic. 
\end{enumerate}
\end{theorem} 

It is quite striking that on surfaces, we obtain a characterization of the SPR property for a Borel homoclinic class that only depends on its MME. We do not know if it can be extended to higher dimension or if there are other such characterizations.

\begin{corollary}\label{thmSPRfromExpRet}
Let $f$ be a $C^{1+}$-diffeomorphism of a surface and let $X$ be a Borel homoclinic class. Then $f|X$ is SPR if and only if there is some $0<\chi<\hTOP(f|X)$ such that, for every $0<\epsilon<\chi$, there is $0<\theta<1$ such that, writing $\Lambda:=\Lambda(\chi,\epsilon,C)$,
    $
    \mu[\tau_\Lambda>n]=O(\theta^n)$ as $n\to\infty$, where $\tau_\Lambda(x):=\inf\{n\geq 0: f^n(x)\in \Lambda\}.
    $
\end{corollary}

Indeed, item (3) above is ensured by Ruelle's inequality \cref{thmRuelleInequ}, hence SPR is equivalent to the existence of some MME for $f|_X$ satisfying item (2) for some $0<\chi<\hTOP(f)$.

\medbreak
\paragraph{Dynamical stochastic processes.} 
Recall that a stochastic process is a family of random variables, i.e., measurable functions defined on some common probability space. Given a function $\psi:M\to\mathbb R$, the dynamical system $(f,\mu)$ gives rise to a stochastic process according to $\psi_n(x):=\psi(x)+\psi\circ f(x)+ \dots +\psi\circ f^{n-1}(x)$ where $\mu$ is the law of $x\in M$. 
We denote the \emph{expectation} by $\mathbb E_\mu(\psi):=\int_M \psi\, d\mu$ and the \emph{asymptotic variance}:
 $$
    \sigma_{f,\mu,\psi}^2 := \lim_{n\to\infty} \frac1n\int_M |f-\mathbb E_\mu(f)|^2\, d\mu.
 $$
We often omit the dependence of this variance on the dynamics $(f,\mu)$ and denote it simply by $\sigma^2_\psi$. This quantity contains a lot of information related, e.g., to the decay of correlations or the thermodynamical formalism (see Theorem 11.13 in \cite{BCS-SPR}).

When $\psi$ is regular, the hyperbolicity of $f$ gives rise to some weak independence of the summands $\psi\circ f^k(x)$ and one can hope that most of the asymptotic properties of $(\psi_n)_{n\ge0}$ are the same as those of a sum of i.i.d. random variables, except in degenerate ("non random") cases. Here we give one characterization of it:

\begin{theorem}[B-Crovisier-Sarig \cite{BCS-SPR}]
Let $f$ be a $C^{1+}$-diffeomorphism and let $X$ be an SPR Borel homoclinic class.
Let $\psi$ be a H\"older function on $X$. Its asymptotic variance $\sigma_{f,\mu,\psi}^2$ vanishes if and only if there is $c\in\mathbb R$ such that, for all $n$-periodic point $q\in X$, $\frac1n\psi_n(q)=c$. 
\end{theorem}

\medbreak
\paragraph{Large deviations principle.}
We obtain the following generalization of a result of Kifer for uniformly hyperbolic diffeomorphisms \cite{Kifer-1990}. Let us give here a simple corollary of \cite[Thm 11.15]{BCS-SPR}.

\begin{theorem}[B-Crovisier-Sarig \cite{BCS-SPR}]\label{thmLDP-diffeo}
Let $f\in\Diff^{1+}(M)$ be SPR with some measure $\mu$ maximizing the entropy.
For every $\alpha>0$ there is $c>0$ such that for all $\psi\in C^\alpha(X)$ with  $\int\psi d\mu=0$ 
Then,  
for every $\epsilon>0$, for every $a>0$ small enough,
 $
   \mu\{x\in M: \psi_n(x)\ge n\cdot a\} \le e^{-(1-\epsilon)a^2/2\sigma^2_\psi}
 $
\end{theorem}

\paragraph{Approximation by Brownian motion.}
We are going to relate the process $(\psi_n)_{n\ge0}$ to a Brownian motion in the sense of the ASIP defined below. This is well-known to allow one to  recover many limit theorems valid for the Brownian motion such as Central Limit Theorems, Law of Iterated Logarithm, Arc Sine Law, or Law of Records.

\medbreak

Recall that a \emph{standard Brownian motion} is a family of functions $B_t:\Omega\to\mathbb R$ $(t\in [0,\infty))$ all defined on the same standard probability space $(\Omega,\mathcal F,m)$ such that: 
(1) $(t,\omega)\mapsto B_t(\omega)$ is measurable;
(2) $B_0\equiv 0$, and $B_t-B_s$ has Gaussian distribution with mean zero and variance $|t-s|$, i.e. $m\{\omega:B_t(\omega)-B_s(\omega)<\tau\}=(2\pi|t-s|)^{-1/2}\int_{-\infty}^\tau e^{-x^2/2|t-s|}dx$;
(3) For each $0<t_1<\cdots<t_{n+1}$, $B_{t_{i+1}}-B_{t_i}$ are independent random variables, i.e., $m(\bigcap_i\{\omega: B_{t_{i+1}}(\omega)-
    B_{t_{i}}(\omega)\in E_i \})=\prod_i m\{\omega: B_{t_{i+1}}(\omega)-
    B_{t_{i}}(\omega)\in E_i\}$ for every Borel sets $E_i\subset\mathbb R$.

\medbreak
\begin{definition}[{\bf ASIP}]\label{def-ASIP}
We say that a stochastic process $(S_n)_{n\ge1}$ satisfies the \emph{almost sure invariance principle (ASIP)} with parameter $\sigma\ge0$ and rate $o(n^\gamma)$ for $0<\gamma<\tfrac{1}{2}$, if there exist two stochastic processes  $(\widetilde{S}_n)_{n\geq 1}$ and $(\widetilde{B}_t)_{t\geq 0}$ defined on a common standard probability space
such that
\begin{enumerate}
\item the stochastic processes $(\widetilde{S}_n)_{n\geq 1}$ and  $(S_n)_{n\geq 1}$ are equal in distribution;
\item $({\widetilde B}_t)_{t\geq 0}$ is a standard Brownian motion;
\item $|\widetilde{S}_n-\sigma {\widetilde B}_n|=o(n^{\gamma})$ a.e. as $n\to\infty$. 
\end{enumerate}
\end{definition}

\begin{remark}
If a stochastic process satisfies the ASIP for some value of $\sigma$ above, this value is then uniquely defined. Indeed, the ASIP gives the convergence of   $S_n/\sqrt{n}$  to a normal distribution, so $\sigma$ is the standard deviation of the limit.
\end{remark}

We obtain the following result, generalizing the work  of Denker \& Philipp \cite{Denker-Philipp} for Anosov diffeomorphisms.

\begin{theorem}\label{t.ASIP-diffeo}
Let $f$ be a $C^{1+}$ diffeomorphism of a closed manifold, let $\mu$ be the MME of $f$. Suppose $\psi:M\to \mathbb R$ is  H\"older with $\int\psi d\mu=0$.
Then $S_n(x):=\psi(x)+\psi(f(x))+\cdots+\psi(f^{n-1}(x))$
satisfies the ASIP 
with $\sigma=\sigma_\psi$. 
 \end{theorem}

This concludes our review of the consequences of the SPR property either globally or on a Borel homoclinic class in arbitrary dimension.

\section{Finding SPR in higher dimensions}
We turn to the question of finding SPR diffeomorphisms in higher dimensions. We already noted that one cannot expect that positive topological entropy would imply SPR in dimension $3$ or more.

\newcommand\wsc{\to}

\subsection{Characterization by Lyapunov exponents}
As on surfaces, our main tool will relie on Lyapunov exponents. Let us consider $f\in\Diff^1(M)$ and try to prove SPR for an  invariant Borel set $X\subset M$ --in our applications it will be either the whole of $M$ or a Borel homoclinic class. 

We are going to generalize \cref{thm-CE-SPR} from surfaces to arbitrary dimension. 

First, we will need some hyperbolicity condition, since SPR requires all measures of large entropy to be $\chi$-hyperbolic for some $\chi>0$. Note that if our diffeomorphism is $C^{1+}$ and admits some Ben Ovadia coding that is SPR, then any sequence of ergodic measures with entropy converging to the top entropy will have lifts that converge to the MME above. Hence the limit below should also be $\chi$-hyperbolic  and ergodic. We will use the following, weaker condition:
\begin{enumerate}
\item[(EH)] \emph{\bf Entropy Hyperbolicity on $X$}:
There is $\chi>0$  such that for all $\mu_n\in \Proberg(f|_X)$ that weakly converge to some $\mu$ on $M$ with  $h(f,\mu_n)\to \hTOP(f|_X)$, there exists a constant $i:=i(\mu)$ such that 
$ \lambda^i(x) >\chi>-\chi>\lambda^{i+1}(x)\text{ $\mu$-a.e.}$
\end{enumerate}

Remembering the proof of \cref{thm-CE-SPR}, we need an extra condition to ensure that the lifts of the measures $\nu_k$ to the Grassmanian bundle by the map defined by the unstable spaces: $x\mapsto E^+_x$ also converge. It turns out to suffice to consider the sum of the positive exponents:
 $$
    \Lambda^+(\mu) = \sum_{i=1}^d \max(\lambda^i(f,\mu),0)
 $$
and require the following continuity:
\begin{enumerate}
\item[(EC)] \emph{\bf Entropy Continuity of $\Lambda^+$ on $X$:}
Suppose   $\mu_n\in \Proberg(f|_X)$ and 
$\mu_n\wsc\mu$ on $M$. If $h(f,\mu_n)\to \hTOP(f|_X)$, then
$
    \lim_{n\to\infty} \Lambda^{+}(\mu_n)=\Lambda^{+}(\mu).
$
\end{enumerate}

Indeed, it follows by the same strategy as in the two-dimensional version (\cref{thmSPRECLEsurf}) that, for a $C^1$-diffeomorphism and an invariant Borel set $X$, (EH) and (EC) together imply SPR.
When $f$ is $C^{1+}$ and $X$ is a Borel homoclinic class, we can use the coding and get the converse 

\begin{theorem}[B-Crovisier-Sarig \cite{BCS-SPR}]\label{thm-EHEC}
Let $f$ be a $C^{1+}$-diffeomorphism and $X\subset M$ be a Borel homoclinic class or the whole of $M$.
Then $f|_X$ is  SPR if and only if it satisfies (EH) and (EC).
\end{theorem}

The rigidity phenomenon still holds:

\begin{theorem}[B-Crovisier-Sarig \cite{BCS-SPR}]\label{thm-CE-SPR}
Let $f$ be a $C^{1+}$-diffeomorphism and $X\subset M$ be an SPR Borel homoclinic class. Then, writing $m$ for the unique MME of $X$ 
 $$
   \forall\mu\in\Prob(f|_X)\quad |\Lambda^+(m)-\Lambda^+(\mu)| \leq C\sqrt{h(f,m)-h(f,\mu)}.
 $$
\end{theorem}

However we do not believe that one can control the individual exponents in general.

\subsection{Partially hyperbolic diffeomorphisms that are SPR}
It turns out that a number of previously studied classes of diffeomorphisms can in fact easily be seen to be SPR with all the consequences (sometimes known by other techniques). We give a two examples (though there are other interesting related works, e.g., \cite{Mongez-Pacifico-2024,marin-poletti-arxiv-2025}).

Let us recall the definition. We refer to \cite{Crovisier-Potrie-Cours} for details and background.

\begin{definition}\label{DefPH}
A diffeomorphism $f:M\to M$ is \emph{partially hyperbolic} if there exists an invariant splitting $TM=E^s\oplus E^c\oplus E^u$  such that: (i) $E^s$ is uniformly contracted; $E^u$ is uniformly expanded, and the splitting is dominated, i.e., there is $\lambda>0,C>1$ such that, for all $x\in M$, for all unit vectors $v^*\in E^*_x$:
 $$
   \forall n\ge0\;  Ce^{\lambda n} \|Df^n.v^s\| \leq \| Df^n.v^c\| \leq C^{-1} e^{-\lambda n}\|Df^n.
  $$
\end{definition}

If $E^c=0$, this is equivalent to Anosov. In general, their dynamics along $E^c$ is arbitrary so one needs to add some assumptions for a fruitful study.

\medbreak
\paragraph{Center fibration into circles.}
We first consider classes of partially hyperbolic $C^{1+}$-diffeomorphisms such that $E^c$ is one-dimensional and tangent to circles:

\begin{theorem}[F. Rodriguez-Hertz, J. Rodriguez-Hertz, Tahzibi, Ur\`es /  Tahzibi/ J. Yang]
Let $f$ be such a diffeomorphism of $\mathbb T^3$. Assume that it is accessible, i.e., for every pair of points $x,y\in \mathbb T^3$, there is a path which is everywhere tangent to $E^u\cup E^s$. Then:
 \begin{itemize}
  \item $\MME(f)$ is finite and either (a) contains a unique measure $m$ with $\lambda^c(m)=0$; or (b) contains more than one hyperbolic measure \cite{RRTU-compact};
  \item in case (b), there are $\chi,\delta>0$ such that every $\mu\in\Proberg(f)$ with $h(f,\mu)>\htop(f)-\delta$ is $\chi$-hyperbolic \cite{Tahzibi-Yang-TAMS-2019}.
 \end{itemize}
\end{theorem}

Item (2) was proved before the discovery of the SPR property using the so-called invariance principle of Avila and Viana \cite{Avila-Viana-IP}. It is now easy to apply \cref{thm-EHEC} and get:

\begin{corollary}
In the setting of the above theorem, $f$ is SPR if and only if it is in case (b).
\end{corollary}

\medbreak
\paragraph{Discretized Anosov Flows.}
We consider another class of partially hyperbolic diffeomorphism with non-compact one-dimensional center leaves that includes perturbations of time-$1$ map of Anosov flows. This class was introduced by Martinchich \cite{Martinchich-JMD-2023}.

\begin{theorem}[B-Crovisier-Poletti-Tahzibi (in preparation)]
Let $f$ be a $C^{1+}$-smooth discretized Anosov flow in the sense of \cite{Martinchich-JMD-2023}. Assume that the underlying flow is topologically transitive and not a suspension.
 \begin{itemize}
  \item either (a) $f$ has a single MME $m$ with $\lambda^c(m)=0$, or (b) it has exactly two MMEs, both hyperbolic, onbe with positive, the other with negative exponent.
  \item in case (b), $f$ is SPR.
 \end{itemize}
\end{theorem}

\section{Questions.}
We present a selection of questions and problems that we believe to be of interest.
 
\subsection{Abundance of SPR diffeos.}
One can be very bold and ask whether SPR might succeed where uniform hyperbolicity failed:

\begin{question}
Do the SPR diffeomorphism contain an open and dense (or generic?) subset of $\Diff^\infty(M^d)$?
\end{question}

This question is probably extremely hard in such generality, or maybe destined to provoke the reader into finding counter-examples. We would be interested by any case (preferably with smoothness at least $C^{1+}$) and expect to find already interesting questions in special classes of surface diffeomorphisms.

\medbreak

Observe that we do not claim that the SPR property is robust: we believe with S. Crovisier that one can construct a zero entropy counter example on any surface. However, the construction is fragile and might be accumulated by robustly SPR diffeomorphisms.

The case where $r=\infty$ and $d=2$ might be tempting in light of our main result \cref{thmSurfaceSPR}, so that we would "only" have to prove:

\begin{question}
Is  (robust) SPR  dense in the set of surface $C^\infty$-diffeomorphisms with zero entropy?
\end{question}

However this seems dangerously close to the $C^\infty$ version of the weak Palis Conjecture (only known in $C^1$, see \cite{Crovisier-ast}): {\sl are the Morse-Smale diffeomorphisms dense in $\{f\in\Diff^\infty(M^2):\htop(f)=0\}$~?}

For the case of finite $r>1$, we note that Burguet \cite{Burguet-Decomposition} has strengthened  Theorem~\ref{thmECLE} and shown that all surface $C^r$-diffeomorphisms with $\htop(f)>\lambda/r$ are also SPR. 

In the case of small entropy, we ask:

\begin{problem} 
Find a  $C^r$-diffeomorphism with $\htop(f)=\lambda/r>0$ for some $1\le r<\infty$ that is SPR but with infinitely many ergodic measures maximizing the entropy.
\end{problem}

\subsection{SPR without coding}
We rely extensively on the coding by Markov shifts for the proof of the consequences of SPR. 

\begin{problem}
How much is it possible to replace the symbolic dynamics by more geometric arguments?
\end{problem}

The (one-sided) SPR coding provides us with a transfer operator having a spectral gap. Intrinsic transfer operators have been developed for uniformly hyperbolic situations only recently a non-uniform setting such as SPR seems very difficult. But other approaches exist, at least for some of the properties. In particular, we ask:

\begin{question}
Can one prove, without using symbolic dynamics, that an SPR surface $C^{1+}$-diffeomorphism: has some   MME?  has only finitely many MMEs? 
Can one establish the equivalence with the continuity of the exponents of \cref{thmECLE-htop-surf}?
\end{question}

The other consequences seem more dependent on the existence of a spectral gap and we only know of a suitable transfer operator in the setting of one-sided Markov shift. Maybe the exponential mixing could still be obtained from more geometric arguments, such as the coupling technique, see \cite{Young-IJM,Dolgopyat-Mostly-2000}.

\subsection{SPR for SRB}

We have mentioned (\cref{remSPRpot}) that one can defined a completely parallel theory of Strong Positive Recurrence with respect to a potential function $\phi$ in order to study equilibrium states beyond measures maximizing the entropy. The more difficult question is to show that this new $\phi$-SPR property is satisfied.

Some cases are easy: if $f\in\Diff^\infty(M^2)$ and $\phi:M\to\mathbb R$ is continuous and has a unique equilibrium state $m$ which moreover has positive entropy, then it is not difficult to see that $p_\phi(f,\nu_k)\to \sup_\mu p_\phi(f,\mu)$ implies that $\nu_k$ converge to $m$ weakly and in entropy, hence in Lyapunov exponents. This implies the $\phi$-SPR property by an adaptation of \ref{thmECLE}.

However, the most interesting case is that of the Sinai-Ruelle-Bowen (or SRB) measures (see, e.g., \cite{Young-SRB-2002}). Those are equilibrium states but with respect to a measurable potential given by the unstable Jacobian: $J^u(x):=-\log|\det Df|E^+_x|$. This potential is \emph{not} continuous on $M$. Therefore the previous analysis breaks down. In fact, it is well-known that some ergodic, hyperbolic SRB measures are not exponentially mixing (see \cite{BCS-SPR} for a statement), hence cannot be $J^u$-SPR.

\begin{problem}
Find sufficient conditions for the SRB measures to be $J^u$-SPR. Is this common or rare? Can it be checked for H\'enon maps for ``good parameters'' using the deep analysis of that fundamental example?
\end{problem}

There are two well-known classes of partially hyperbolic diffeomorphisms (see \cref{DefPH})  for which this notion of $J^u$-SPR is very efficient.  
These classes are defined by requiring all their $u$-Gibbs states, i.e., their invariant measures $m^u$ absolutely continuous along the strong unstable foliation, to satisfy:

$\bullet$ (Mostly contracting) all the Lyapunov exponents of $m^u$ along $E^c$ are negative;

$\bullet$ (Mostly expanding) all the Lyapunov exponents of $m^u$ along $E^c$ are positive.

(These are modifiations of the classes introduced by \cite{Alves-Bonatti-Viana-2000,Bonatti-Viana-IJM-2000}, see \cite{andersson-vasquez-2018}).
Using \cite{Hu-Hua-Wu-2017,Yang-adv-2021} one can show that in both mostly contracting and mostly expanding diffeomorphosms are $J^u$-SPR.

\subsection{Stronger forms of SPR}
We have seen that SPR is powerful and satisfied by many systems. A natural question is whether these systems satisfy in fact an even stronger property.

\medbreak
\paragraph{Spectral decomposition and entropy at infinity}
We have shown in \cite{BCS-MME} that surface $C^r$-diffeomorphisms admit a \emph{spectral decomposition} generalizing that of Smale. In particular, if $H_1,H_2,\dots$ are the homoclinic classes of measures, we have seen that $\limsup_{n\to\infty} \hTOP(H_n)\leq\lambda/r$. This suggests:

\begin{question}
Do surface diffeomorphism satisfy the SPR property with an entropy threshold $h_0<\hTOP(f)$ independent of $\epsilon>0$? Can this threshold be taken arbitrarily close to $\lambda/r$?
\end{question}

Another point of view is to consider the case of Markov shifts. For those one can define an \emph{entropy at infinity} (see, e.g., \cite{Buzzi-AIF,Iommi-Todd-Velozo-2022}) which is the least possible upper bound on the entropy of invariant measures that give an arbitrarily small measure to an arbitr	arily large finite set of vertices. 

\begin{question}
Can one define an entropy at infinity for surface $C^r$-diffeomorphisms? Is it bounded by $\lambda/r$?
\end{question}

We note that D. Yang and C. Luo \cite{luo-yang-arxiv-2025} have shown another type of control for measures with small, positive entropy. They get a lower bound on the size of unstable manifolds for a set of lower bounded measure.

\paragraph{Periodic points and zeta functions}
We already mentioned Artin-Mazur zeta functions, see \eqref{eqZeta}. They have infinite meromorphic extension for uniformly hyperbolic diffeomorphsms and none at all for many surface $C^\infty$-diffeomorphisms by a construction of Kaloshin \cite{Kaloshin-zeta}. 

For irreducible SPR Markov shifts, the zeta function is likewise not always defined, since there may be arbitrary numbers of small loops very far in the graph.

However, Burguet \cite{Burguet-Periodic} has shown that for surface $C^\infty$-diffeomorphism this problem disappears if one restricts oneselves to the periodic orbits that are $\chi$-hyperbolic (i.e., hyperbolic with exponents bounded away from zero). Then:
 $$
   \lim_{n\to\infty} \frac1n\log\#\operatorname{Fix}^\chi_n(f) = \htop(f)
   \text{ where } \operatorname{Fix}^\chi_n(f):=\{x\in M:f^n(x)=x,\; \text{ is $\chi$-hyperbolic}\}.
  $$

\begin{question}
Let $f\in\Diff^\infty(M^2)$ with $\htop(f)>0$. Fix $0<\chi<\htop(f)$ and consider:
 $$
   \zeta^\chi(z) := \exp\left(\sum_{n\ge1} \frac1n\#\operatorname{Fix}^\chi_n(f) \right)
 $$
Does it have a meromorphic extension beyond its obvious radius of converge $|z|=\htop(f)$? Does the maximal radius of meromorphic extension have a dynamical meaning?
\end{question}

In the case of SPR Markov shift, there is a local zeta function whose meromorphic radius is lower bounded by the entropy at infinity.

\subsection{SPR for other settings}

\paragraph{Flows.}
Flows admit a natural notion of Pesin block (at least in the nonsingular case, one can write $TM=E\oplus\mathbb R.X\oplus F$ with $X$ the flow direction). Therefore one can write down a natural definition of a SPR flow.  \cite{Zang-arxiv-2025} has extended Burguet's decomposition from surface diffeomorphisms to $3$-dimensional flows.

\begin{problem}
Develop the analoguous SPR theory for $C^\infty$ nonsingular flows on $3$-dimensional flows. Decide which properties of SPR diffeomorphisms carry over to this new situation (probably not exponential mixing).
\end{problem}

\paragraph{Noncompact manifolds}
One can understand the SPR property as a \emph{tightness} property, controlling the escape of measure away from the parts of the dynamics with uniform hyperbolic estimates. It would seem natural to use such condition to control actual lack of compactness. 

\begin{problem}
Define and study an SPR property using compact Pesin blocks for diffeomorphisms (and flows) on noncompact manifolds.
\end{problem}

One of the motivation would be to interact with the theory of the geodesic flows on SPR manifolds that has been developed by several authors including Barbara Schapira, S. Tapie, and S. Gou\"ezel \cite{Gouezel-Schapira-Tapie-2023}.

\section*{Acknowledgments.}
I feel greatly indebted to Sylvain Crovisier and Omri Sarig for 10 years of a very stimulating collaboration. I thank Sylvain Crovisier for various discussions about this text.

\bibliographystyle{siamplain}
\bibliography{Buzzi-ICM_references}
\end{document}